\documentclass[11pt,twoside,reqno]{amsart}
\pdfoutput=1

\usepackage{amsxtra}
\usepackage[dvipsnames]{xcolor}
\usepackage{amsopn}
\usepackage{amsmath,amsthm,amssymb}
\usepackage{hyperref}
\usepackage{array,multirow,makecell}
\setcellgapes{1pt}
\makegapedcells
\newcolumntype{R}[1]{>{\raggedleft\arraybackslash }b{#1}}
\newcolumntype{L}[1]{>{\raggedright\arraybackslash }b{#1}}
\newcolumntype{C}[1]{>{\centering\arraybackslash }b{#1}}
\usepackage{lscape}
\newtheorem{teo}{Theorem}[section]
\newtheorem{prop}{Proposition}[section]
\newtheorem{corol}{Corollary}[section]
\newtheorem{lemm}{Lemma}[section]
\newtheorem{remark}{Remark}[section]
\newtheorem{ex}{Example}[section]
\newtheorem{defin}{Definition}[section]
\renewcommand{\a}{\alpha}

\newcommand{\beq}{\begin{equation}}
\newcommand{\eeq}{\end{equation}}
\newcommand{\bqn}{\begin{eqnarray}}
\newcommand{\eqn}{\end{eqnarray}}
\newcommand{\bqne}{\begin{eqnarray*}}
\newcommand{\eqne}{\end{eqnarray*}}

\newcommand{\R}{{\mathbb R}}

\newcommand{\C}{{\mathbb C}}

\makeatletter
\newcommand\justify{%
  \let\\\@centercr
  \rightskip\z@skip
  \leftskip\z@skip}
\makeatother

\definecolor{MyBlue}{RGB}{0,0,255}

\begin{document}

\title{Locally Conformal SKT  structures}

\author{Bachir Djebbar}
\address[B. Djebbar]{Department of Mathematics,  Faculty of Mathematics and Informatics,  University of Science and Technology of Oran Mohamed Boudiaf El Mnaouar,   Bir El Djir,  Oran,  31000 Algeria}
\email{bachir.djebbar@univ-usto.dz}

\author{Ana Cristina Ferreira}
\address[A. C. Ferreira]{Centro de Matem\'atica, Universidade do Minho,  Campus de Gualtar,  4710-057 Braga,  Portugal}
\email{anaferreira@math.uminho.pt}

\author{Anna Fino}
\address[A. Fino]{Dipartimento di Matematica “G. Peano”, Universita` di Torino, Via Carlo Alberto 10, 10123 Torino, Italy  \& Department of Mathematics and Statistics Florida International University
Miami Florida, 33199, USA } \email{annamaria.fino@unito.it,  afino@fiu.edu}
	
\author{Nourhane Zineb Larbi Youcef}
\address[N. Z. Larbi Youcef]{Department of Mathematics,  Faculty of Mathematics and Informatics,  University of Science and Technology of Oran Mohamed Boudiaf El Mnaouar,   Bir El Djir,  Oran,  31000 Algeria}
\email{zineb.larbiyoucef@univ-usto.dz}

\date{\today}



\begin{abstract}
A Hermitian metric on a complex manifold is called SKT (strong K\"ahler with torsion) if the Bismut  torsion  $3$-form $H$  is closed. As the conformal generalization of the SKT condition,  we introduce a new  type of Hermitian structure,  called \emph{locally conformal   SKT} (or shortly LCSKT).  More precisely,  a Hermitian structure $(J,g)$  is said to be   LCSKT if there exists a closed  non-zero $1$-form $\alpha$ such that $d H = \alpha \wedge  H$.  In the paper we consider  non-trivial LCSKT structures, i.e. we  assume that $d H \neq 0$  and we study their existence on Lie groups and their compact quotients by lattices.

In particular,  we  classify  6-dimensional  nilpotent Lie algebras  admitting a  LCSKT structure and we  show that, in contrast to the SKT case,  there exists a $6$-dimensional  $3$-step nilpotent Lie algebra admitting  a non-trivial LCSKT structure.  Moreover, we show  a characterization of even dimensional  almost abelian  Lie  algebras  admitting  a non-trivial LCSKT structure, which allows us to construct explicit examples of $6$-dimensional unimodular  almost   abelian  Lie algebras  admitting  a non-trivial LCSKT structure.  The  compatibility between  the LCSKT and  the balanced condition is  also discussed,  showing  that a Hermitian structure on a 6-dimensional nilpotent or a  $2n$-dimensional almost abelian
 Lie  algebra  cannot be simultaneously LCSKT and balanced, unless it is K\"{a}hler.
\end{abstract}

\maketitle

\section{Introduction}

Among Hermitian metrics, an important class  is given by the strong K\"{a}hler with torsion (SKT) (or pluriclosed)  which may be characterized by  the condition $\partial \overline \partial \Omega=0,$ where $\Omega$ is the fundamental form.  An equivalent geometrical meaning is given in  terms of  the  Bismut (or Strominger) connection  $\nabla^B$ \cite{Yano, Bismut, Stro86}, which is the unique connection compatible with the Hermitian structure $(J,g)$ (i.e. $\nabla^B g=0 =\nabla^B J$) such that its  torsion $H$  is a totally skew symmetric tensor.  The torsion can therefore be  identified with a 3-form which is given by $H  = Jd \Omega$ or $H = d^c \Omega$ with $d^c=i(\overline\partial - \partial)$.  The Hermitian structure is called SKT if $H$ is $d$-closed i.e. $d H=0$.

Initially,  SKT metrics appeared in theoretical physics,  applied widely in 2-dimensional sigma models \cite{GHR, HP},  having also emerged from   superstring theories with torsion \cite{Stro86}. From a mathematical point of view,  they are deeply related to generalized K\"{a}hlerian geometry \cite{GHR,GM2, FP}.

Concerning the existence problem for SKT metrics on Hermitian manifolds, there are no general conditions.  All the results obtained in the literature indicate that the existence of an SKT metric depends on the type of manifold considered and its dimension.

For instance, in dimension four, any compact complex surface is always an SKT manifold, since each conformal class admits a unique SKT standard metric in the sense of Gauduchon \cite{G84}.  It is proven in \cite{MS}, that  on a  4-dimensional
 unimodular solvable Lie group   a  left-invariant  Hermitian  structure is  SKT and a classification of  4-dimensional SKT  Lie algebras is  also given.

However, for higher dimensions the existence of  a left-invariant  SKT   structure  on  unimodular  (non-compact) Lie groups is not guaranteed anymore,  unless certain selective conditions are required.

For example, by \cite{FPS}  only 4  isomorphism classes of  6-dimensional (non-abelian)  nilpotent  Lie algebras  admit an SKT structure,  see also \cite{Ug} for  a classification up to equivalence of complex structures.
An interesting property is that the existence of an SKT structure  $(J,g)$ on nilpotent Lie algebras of dimension 6 depends only on the complex structure $J$.

The 8-dimensional SKT nilpotent Lie algebras  are also classified in \cite{EFV}.  In a recent article \cite{AN}, it was shown that every  nilpotent Lie algebra  admitting a SKT structure is  at most 2-step nilpotent.


For the  more general class of  SKT solvable  Lie algebras, the only complete classification is in dimension four \cite{MS}.
Classifications of SKT solvable Lie algebras  of  six and higher dimensions are established  only under some extra conditions.  For example, in \cite{FOU} the classification of  six-dimensional SKT   unimodular solvable Lie algebras admitting complex structures  with non-zero closed  holomorphic $(3,0)$-form is given.   Moreover, a general  characterization of SKT structures on $2n$-dimensional almost abelian Lie algebras is established in \cite{AL}, and  a classification for the 6-dimensional case is given  in \cite{FP}.  The case of almost nilpotent solvmanifolds whose nilradical has one-dimensional commutator was considered in \cite{FP22}.  Also,  a  wide range of two-step solvable Lie groups admitting a left-invariant SKT structure  was  classified in \cite{FAS}, the same authors have recently obtained  the full classification of left-invariant SKT structures on two-step solvable Lie groups in dimension six,  \cite{FS22}.

In this article  we introduce a new  type of Hermitian structure,  called \emph{locally conformal SKT} (shortly LCSKT), as a generalization of  SKT structures, inspired by ideas of  locally conformal K\"{a}hler (LCK) geometry \cite{V1, V2, AU}  and we study the existence of this type of structure  on Lie groups  and their compact quotients by lattices.

Let  $H$ be  the Bismut  torsion $3$-form  of a Hermitian  manifold $(M, J, g)$.  In  Section \ref{Sec: general}, we  define  an   LCSKT Hermitian manifold   $(M, J, g)$ as a Hermitian manifold, such that the  3-form  $H$  satisfies   the condition  $d H=\a \wedge H$, where   $\a$  is a  non-zero $d$-closed   $1$-form.    Notice that  $\alpha$  is not necessarily unique and,  if  $\alpha =0,$ then  $(J, g)$ is SKT. In the case where the 3-form $H$ is non-degenerate, i.e.  $\ker \, H  \neq \{ 0 \}$,  we prove that the definition of an  LCSKT structure is equivalent to  the $d$-closeness of each local 3-form $(e^{-f_i} H)\vert_{U_i}$, such that  $\{U_i\}$ is an open cover in $M$ and   $\{ f_i \}$ a family of  smooth real functions $f_i:U_i \rightarrow \R$   (Propositions \ref{Prop: localforms-if} and  \ref{Prop: localform-ifonlyif}).

In  Section \ref{Sec: nilpotent}, we study the existence of  LCSKT structures $(J,g)$  on  six-dimensional  nilpotent Lie algebras.   In  particular we provide  a complete classification of nilpotent Lie algebras admitting an  LCSKT structure (Theorem \ref{Thm: nilpotent-main} and  Corollary \ref{Cor: nilpotent-main}), showing that up to isomorphism only one  admits a non-trivial LCSKT structure.  In contrast to the SKT case this nilpotent  Lie algebra  is $3$-step.  Moreover, we prove that  a 6-dimensional nilpotent  Lie algebra,  admitting an LCSKT structure and a balanced metric, has to be abelian.

 We devoted  Section \ref{Sec: almost-abelian} to the study of the existence of LCSKT structures on  almost abelian  Lie algebras, i.e. on solvable Lie algebras admitting an abelian ideal of codimension one, determining  a characterization of LCSKT almost abelian Lie algebras in  arbitrary dimension
 $(2n \geq 6)$. In particular, we  construct some explicit examples of unimodular  (non-nilpotent) almost abelian Lie algebras admitting  LCSKT  structures, which allow to give examples of $6$-dimensional LCSKT  solvmanifolds. Moreover, we prove that  on an almost abelian
 Lie  algebra  a Hermitian structure cannot be simultaneously LCSKT and balanced, unless it is K\"ahler.

\section{Locally conformal SKT structures}\label{Sec: general}

Let $(M,J,g)$ be a $2n$-dimensional Hermitian manifold, such that $J$ is a complex structure on $M,$ which is orthogonal relative to the Riemannian metric $g$. The 2-fundamental form $\Omega$ is given by $\Omega(X,Y)=g(JX,Y),$ for any vector fields  $X, Y$ on $M$. If $d \Omega =0$, the metric $g$ is said to be {\em K\"ahler} and the Levi-Civita connection $\nabla^g$ is Hermitian, i.e.  $J$ and $g$ are both parallel  with respect to $\nabla^g$.

In \cite{G97},  Gauduchon proved that  on  $(M, J, g)$ there exists an affine line $\{\nabla^t\}_{t \in \R}$ of canonical Hermitian connections, passing through the Chern connection $\nabla^C$ and the Bismut connection $\nabla^B$. The  connections $\nabla^t$ preserve both  $g$ and $J$ and they  are completely determined by their torsion. In particular, the Chern connection $\nabla^C$ is the unique Hermitian connection whose torsion has trivial $(1,1)$-component and the Bismut connection $\nabla^B$ is the unique Hermitian connection with totally skew-symmetric torsion $H$ (\cite{Bismut}).

 Since  $\nabla^B$ and $\nabla^C$ preserve the Hermitian structure $(J, g)$,  they induce unitary connections on  the anticanonical bundle $K^{-1}$,  with curvatures respectively  $i \rho^C$ and $i \rho^B$, where
\begin{equation*}
\rho^C(X,Y):= \frac{1}{2} \sum_{i=1}^{2n}g(R^C(X,Y)e_i,Je_i)
\end{equation*}
 is the Ricci form of  $\nabla^C$  and the Ricci form $\rho^B$  of $\nabla^B$  is defined in a similar way.  Here $\{e_i\}_{i=1}^{2n}$ is a local orthonormal frame of $TM$ and for the curvature we use the following convention: $$
 R^C(X,Y) Z = [\nabla^C_X,\nabla^C_Y ] Z -\nabla^C_{[X,Y]} Z.
 $$
The two Ricci forms $\rho^B$ and $\rho^C$ are related by the relation (see\cite[(2.7)]{AI})
\begin{equation}\label{Eq: Ricci-forms}
    \rho^C = \rho^B + (n-1)dJ\theta,
\end{equation}
where
$\theta$ is the  so-called {\em Lee form} $\theta$. The $1$-form  $\theta$    is
defined as the  trace of the torsion of the Chern connection  $\nabla^C$
$$
\theta = \frac{1}{(1-n)}Jd^{\dag}\Omega,
$$
where $d^{\dag}$ is the formal adjoint of $d$ with respect to $g$, or quivalently, it  is the unique $1$-form such that \begin{equation*}
d\Omega^{n-1} = (n-1)\theta \wedge \Omega ^{n-1}.
\end{equation*}

Concretely,  the  Bismut torsion $3$-form  has the  following expression
$$H(X,Y,Z)= d^c \Omega (X, Y, Z) = -d\Omega(JX,JY,JZ),$$
for every vector fields $X,Y,Z$ on $M$,  where $d^c=i(\overline \partial - \partial)$ is the real Dolbeault operator associated to the complex structure $J$.

\begin{remark} Note that $H$ is a real $3$-form of type   $(2,1) + (1,2)$ and  we can write it as
$$H = H^{(2,1)} + H^{(1,2)}, $$
where $H^{(2,1)}= - i \partial \Omega$ and  $H^{(1,2)} = \overline{H^{(2,1)}}$.
\end{remark}

If the 3-form $H$ is $d$-closed i.e. $dH = d d^c \Omega =0$,  the Hermitian  metric $g$  is  called \emph{strong K\"{a}hler with torsion} (SKT)  (or \emph{pluriclosed}).  Other types of Hermitian metrics can be defined in terms of the Lee form $\theta$.
A  Hermitian manifold  $(M,J,g)$ is called \emph{balanced} or \emph{co-symplectic},  respectively  \emph{locally conformal balanced} (LCB), if and only if the Lee form is vanishing ($\theta=0$),  respectively  closed.  Note that for $n=2$ a balanced metric is automatically K\"ahler.
 Moreover, if $d\Omega=\theta \wedge \Omega$ with closed Lee form $\theta$, the Hermitian manifold is said \emph{locally conformal K\"{a}hler} (LCK) \cite{V1}.

By imposing a similar relation  for the Bismut  torsion $3$-form, we  can introduce the following definition.

\begin{defin}\label{Def: LCSKT}
A Hermitian  structure  $(J,g)$  on a $2n$-dimensional  complex manifold $(M,J)$ with fundamental 2-form $\Omega$ is called a locally conformal SKT  (LCSKT for brevity)  structure if and only if there exists a $d$-closed  (non-zero) 1-form $\alpha$ on $M$ such that  $$dH=\alpha \wedge H,$$
where $H=d^c \Omega$.  A manifold $(M, J,g)$ is said to be LCSKT, if it admits an LCSKT structure.
\end{defin}

If $d H =0$ then $(J,g)$ is an SKT  structure  and we will call  it  a  \emph{trivial LCSKT structure}.

\begin{remark} Note that a Hermitian structure $(J,g)$ with $g$ conformal to an SKT metric $\tilde g$  is not automatically an LCSKT structure.
Indeed, if $\tilde \Omega$ is the fundamental form associated to $(J,  \tilde g)$ and $g = e^f  \tilde g$, then $\Omega = e^{f} \tilde \Omega$ and
$$
d^c \Omega  =   e^f  (J df ) \wedge \tilde \Omega + e^f  d^c  \tilde \Omega
$$
and then
$$
d d^c \Omega = df \wedge d^c \Omega + e^f d(Jdf) \wedge \tilde \Omega - e^f (J df) \wedge d \tilde \Omega.
$$
\end{remark}

\begin{remark}
If  $(J, g)$ is  a balanced Hermitian structure on a complex manifold $(M, J)$  of complex dimension  $n \geq 2$,   then by \cite[(2.13)]{AI} we have    $$<dd^c  \Omega,  \Omega^2> =  - 2 |d \Omega|^2.
$$
Therefore,  $g$   is SKT if and only if  it  is K\"ahler. If we impose the  LCSKT condition i.e. $d d^c \Omega = \alpha \wedge d^c \Omega$,  it is not  clear in general if   one   still  has  $d \Omega =0$, but we will see that for particular classes of examples the two conditions are complementary.
\end{remark}

\begin{prop}\label{Prop: localforms-if}
 Let $(M,J,g)$ be an LCSKT  manifold.  Then $M$ has an open cover $\{U_i\}$ and a family $\{f_i\}$ of smooth functions $f_i: U_i\rightarrow \R$  $(f_i \in \C^{\infty} (U_i))$ such that each local 3-form $(e^{-f_i} H)_{\vert U_i}$ is $d$-closed.
\end{prop}
\begin{proof}
 Since $\a$ is closed then it is locally exact, i.e. for every $p \in M$ there is a neighborhood $U_i$ such that $\a =df_i$ for some function
$f_i: U_i \rightarrow \R$. Then, in $U_i$, we have
$$d(e^{-f_i} H)_{\vert U_i} = e^{-f_i} (dH - df_i \wedge H)=e^{-f_i }(dH- \a \wedge H)=0.$$
\end{proof}
The converse of the previous proposition is not always verified because of the eventual non-degeneracy  of $H$.  For this reason, we introduce the notion of  kernel $(\ker\omega)$ of a differential  form $\omega$ on $M$,
$$\ker\omega = \{X \in  \Gamma(TM) |\, \iota_X \omega =0\}$$
where  $\iota_X\omega$ is the interior product  of the differential form $\omega$ by the vector field $X$.  If $\ker\omega = 0$, then $\omega $ is said to be non-degenerate,  ortherwise $\omega$ is said to be degenerate.

\begin{prop}\label{Prop: localform-ifonlyif}
Let $(M,J,g)$ be a Hermitian manifold with fundamental 2-form $\Omega$ such that $H=d^c \Omega$ is non-degenerate. Suppose that  $M$ has an open cover $\{U_i\}$ and a family $\{f_i \}$ of smooth functions $f_i: U_i \rightarrow \R$ such that each local 3-form $(e^{-f_i} H)_{\vert U_i }$ is $d$-closed, then $(M,J,g)$ is an LCSKT manifold.
\end{prop}
\begin{proof}
From the condition $d(e^{-f_i} H)\vert _{U_i}=0$, we have that, in $U_i$,
$dH= df_i \wedge H$. Then, for  every point $p\in M$ there is a neighborhood $U_i$ and a closed one-form $\a_i=df_i$ such that
$dH=\a_i \wedge H$. Let $U_i$ and $U_j$ be two such neighborhoods. Then on $U_i \cap U_j$, we get $(\a_i- \a_j) \wedge H=0$. Since $\ker H=0$, it follows that $\a_i=\a_j$ and so $\a$ is a globally defined closed one-form such that  $dH= \a \wedge H$. Thus, $M$ is LCSKT.
\end{proof}

We have adopted Definition \ref{Def: LCSKT} as the base definition of LCSKT  manifolds  rather than  the property stated in the Proposition \ref{Prop: localforms-if}, in order  to cover the general case, including degenerate 3-forms $H$.

\medskip

\begin{remark}  Note  that, given an    LCSKT manifold $(M, J, g, \Omega)$,  the Chern-Ricci flow
(see for instance \cite{TW})
$$
\left \{
\begin{array}{l}
\partial_t \Omega (t) = - \rho^C (\Omega (t)),\\[2pt]
\Omega (0) = \Omega,
\end{array}
\right.
$$
  preserves the LCSKT condition. Here $\rho^C (\Omega (t))$ denotes the Chern Ricci form of $\Omega(t)$.
Indeed,  the solution  $\Omega(t)$  of the previous  flow  is of the form
$\Omega(t)  = \Omega - t  \rho^C + i \partial \overline  \partial \varphi(t)$,
where  $\rho^C$ is the Chern  Ricci form of $\Omega$ and $\varphi(t)$ solves  the system $(2.1)$ in  \cite{TW}, and so   $d \Omega(t) =  d  \Omega$.

Since $J$  does not evolve along  the flow, we have that
$d^c \Omega (t) = d^c \Omega$.
Thus  $d d^c \Omega (t) = d d^c \Omega$ and the  LCSKT condition is preserved.
\end{remark}

In the paper we will study the existence of invariant LCSKT structures  on  compact  manifolds given by the quotients $\Gamma \backslash G$ of  simply connected Lie groups $G$ by co-compact discrete subgroups  $\Gamma$.

 By invariant LCSKT structure $(J, g)$ on $M = \Gamma \backslash G$  we mean a structure induced by a left-invariant one on $G$ or equivalently  by an LCSKT structure on its Lie algebra $\frak g$.     Therefore we  can    study the existence of invariant  LCSKT structures  on $M$  working at the level of the Lie algebra $\frak g$ of $G$. Furthermore, note that  the form $\alpha$ cannot be  exact (see the proof of Proposition 4.6 in \cite{OOS}).

We will now  shortly introduce the definition of SKT and LCSKT structures on Lie algebras.
Recall that an almost   complex structure  $J$ on $2n$-dimensional real Lie algebra $\frak g$ is defined as an endomorphism of $\frak g$ such that $J^2 = -\mathrm{Id}_{\frak g} $.  If $J$ is  integrable, i.e. if
$$N(X,Y) = [JX,JY] - J[JX,Y] - J[X,JY] - [X,Y] = 0,$$
 for any $X,Y \in \frak g$, then $J$ is called a complex structure on $\frak g$ and  the $i$-eigenspace  $\frak g^{1,0}$ of $J$ in $\frak g^{\C} :=  \frak g  \otimes_{\mathbb R} \frak g$ is a complex subalgebra of $\frak g^{\C}.$   If  $\frak g^{1,0}$ is  an abelian subalgebra  of $\frak g^{\C}$, or equivalently $[JX,JY]=[X,Y],$ for all $X,Y \in \frak g$, then  the  complex structure  $J$ is said to be {\em abelian}. When $\frak g^{1,0}$ is a complex ideal we say that $J$ is bi-invariant, i.e. $J[X,Y ] = [JX,Y ]$, for every $X, Y \in \frak g$.

A   Hermitian metric  $g$ on $(\frak g, J)$ is  a (positive definite) inner product    which is  $J$-orthogonal, i.e. such that $g(JX ,  JY) = g ( X, Y)$, for every $X, Y \in \frak g$.
Let  $\Omega(\cdot , \cdot):=g(J \cdot , \cdot)$   be the  associated fundamental form and $H = d^c \Omega$  be the torsion Bismut form.  Then the  Hermitian structure  $(J, g)$ is SKT  if $dH =0$. Moreover, $(J,g)$  is  LCSKT if there is a closed 1-form $\alpha$ on $\mathfrak{g}$ such that
 \[   dH = \a \wedge H.\]

\begin{remark}  By using that the compact quotient $M = \Gamma \backslash G$  admits a bi-invariant volume form and the invariance of $J$ one can show that  if $(M, J)$ admits a balanced (resp. SKT) metric, then $M$ admits  an invariant balanced (resp. SKT) metric defined by $\Omega_{inv}$ (see \cite{FG,Ug}).
This follows from the fact that given any covariant k-tensor field $T$ on $\Gamma \backslash  G = M$ one can define a  covariant  k-tensor $T_{inv}$  on $\frak g$ as
$$
T_{inv} (X_1, \ldots, X_k) = \int_{p \in M} T_p (X_1 \vert_p, \ldots, X_k \vert_p) \nu, \qquad \mbox{for } X_1, \cdots, X_k \in \frak g
$$
and where $X_i\vert_p $ denotes the value at $p\in M$ of the projection on $M$ of the left-invariant vector field on $G$ defined by $X_i$.

In the case of LCKT structures we can apply symmetrization under the assumption that the closed $1$-form $\alpha$ is invariant.  Concretely,
suppose that $M=\Gamma\backslash G$ is equipped with an invariant complex structure $J$.  Let $g$ be a $J$-Hermitian metric (not necessarily invariant) with fundamental 2-form $\Omega$. If $(J, g)$ is an LCSKT structure with \emph{invariant} closed 1-form $\alpha$ then $(J, g_{inv})$ is an invariant LCSKT structure.  Indeed,  using similar arguments as in \cite{FG, Ug}, we have
$$d(d^c(\Omega_{inv})) = d J d(\Omega_{inv}) = (d J d(\Omega))_{inv} = (d(d^c(\Omega)))_{inv} = (\alpha \wedge d^c(\Omega))_{inv} = \alpha \wedge d^c(\Omega)_{inv}$$
 and our claim follows.
 \end{remark}

\section{Invariant LCSKT structures on $6$-dimensional nilmanifolds}\label{Sec: nilpotent}

In this section we study the existence of invariant  LCSKT structures on  nilmanifolds, i.e. on     compact  quotients $M = \Gamma/G$, where $G$ is a connected and simply connected nilpotent  Lie group and $\Gamma$ is a lattice in $G,$ i.e. a discrete co-compact subgroup $\Gamma \subset G$.

We recall that a  Lie algebra $\frak g$  is {\em nilpotent}  if  its  lower central series $\{ \frak g^i\}$ terminates, i.e. if $\frak g^k = \{ 0 \}$,
for some $k \in {\mathbb N}$, where
$$
\frak g^0 = \frak g, \quad \frak g^j = [\frak g^{j - 1}, \frak g], \, j \geq 1.
$$
In his work \cite{S}, Salamon proves  that   the  existence of a complex structure  $J$  on a $2n$-dimensional nilpotent Lie algebra $\frak g$  is equivalent to the existence of  a basis of $(1,0)$-forms  $\{\omega^j \}_{j=1}^n$ satisfying the following complex structure equations
\begin{center}
$d\omega^1 = 0,$ and $d\omega^j \in \mathcal{I}(\omega^1,...,\omega^{j-1})$ for $j=2,...,n$,
\end{center}
where $\mathcal{I}(\omega^1,...,\omega^{j-1})$ is the ideal in  $\Lambda^* \frak g_{\mathbb{C}}^*$ generated by $\{\omega^1,...,\omega^{j-1} \}$.

A complex structure $J$ on a nilpotent Lie algebra $\frak g$ is called \emph{nilpotent}  in the sense  of \cite{CFGU00},  if there  exists  a basis $\{\omega^j\}_{j=1}^n$ of $(1,0)$-forms  satisfying
\begin{center}
 $d\omega^1=0$ and $d\omega^j \in \Lambda^2 \langle \omega^1,...,\omega^{j-1},\omega^{\overline 1},...,\omega^{\overline{j-1}}\rangle$ for  $j=2,..,n.$
 \end{center}
One can easily check that all abelian complex structures are necessarily nilpotent and that in this case  $d \omega^j$ are of type $(1,1)$.

Nilpotent Lie algebras of dimension $4$ and $6$  admitting a complex structure  have been classified in \cite{S}, with detailed list up to isomorphism in \cite[Theorem 8]{Ug} for the 6-dimensional case, see also \cite{COUV} for a  classification up to equivalence of the complex structures.
In particular,  Ugarte \cite[Proposition 2]{Ug}  showed
the  following   result.
\begin{prop}[\cite{Ug}]\label{Prop: Ugarte}
Let $J$ be a complex structure on a  $6$-dimensional nilpotent Lie algebra $\mathfrak{g}$.
\begin{enumerate}
 \item[(a)]  If $J$  is non-nilpotent, then there is a basis $\{\omega^j\}_{j=1}^{3}$ of $(1,0)$-forms such that
\begin{equation}\label{Eq: J non-nil}
\left \{  \begin{array}{lcl}
 d \omega^1 &=&0,\\[3pt]
 d \omega^2 &=& E\,  \omega^{13} + \omega^{1\overline 3},\\[3pt]
 d \omega^3 &=& A \,  \omega^{1\overline 1 }+ i b \, \omega^{1 \overline 2} - i b  \overline E \, \omega^{2 \overline 1},
 \end{array}
 \right.
\end{equation}
 where $A, E \in \C$ with $| E | = 1,$ and $b \in \R-\{0\}$.
\item[(b)] If $J$ is nilpotent, then there is a basis $\{\omega^j\}_{j=1}^{3}$ of $(1,0)$-forms   satisfying
\begin{equation}\label{Eq: J nil1}
    \left \{  \begin{array}{lcl}
d \omega^1 &=& 0,\\[3pt]
d \omega^2 &= &\epsilon \omega^{1 \overline 1},\\[3pt]
d \omega^3 &= & \rho \, \omega^{12} + (1 - \epsilon) A \, \omega^{1\overline 1} + B \,  \omega^{1 \overline 2}+ C \,  \omega^{2\overline 1} + (1 - \epsilon) D \,  \omega^{2 \overline 2},\\[2pt]
\end{array}
\right.
\end{equation}
where $A, B, C, D \in \C,$ and $\epsilon, \rho \in \{ 0, 1 \}$.
\end{enumerate}
\end{prop}
The abelian complex structures  correspond to the case  $\rho=0$ in the complex  structure equations \eqref{Eq: J nil1}. A  bi-invariant complex structure is automatically nilpotent and corresponds to $\epsilon = A=B=C=D =0$  in \eqref{Eq: J nil1}.

The following Lemma provides a further reduction of the complex structure equations on 2-step  nilpotent Lie algebras (for $\epsilon=0$) (see Proposition 10 and  Lemma 11 in \cite{Ug}).
\begin{lemm}[\cite{Ug}]\label{Lem: epsilon=0}
Let $J$ be a complex structure on a 2-step nilpotent Lie algebra  $\frak g$ of dimension 6 with first Betti number $ b_1 (\frak g) \geq  4$. If $J$ is not bi-invariant, then there is a basis $\{\omega^j\}_{j=1}^3$ of  $(1,0)$-forms such that
\begin{equation}\label{Eq: J nil0}
\left\{  \begin{array}{l}
d \omega^1 = d\omega^2 =0,\\[3pt]
d \omega^3 = \rho \omega^{12} + \omega^{1\overline 1} + B\omega^{1 \overline 2} + D \omega^{2\overline 2},
\end{array}
\right.
\end{equation}
where $B, D\in \C,$ and $\rho \in \{0,1\}.$
\end{lemm}

\begin{remark}
The Lie algebras in Lemma \ref{Lem: epsilon=0} are those for which $J$ is nilpotent with $\epsilon = 0$
 and at least one of $A,B,C,D$ not zero in structure equations (\ref{Eq: J nil1}). \end{remark}

Henceforth we will be assuming that $J$ is not a bi-invariant complex structure on $\frak g$. The bi-invariant case will be dealt with in Remark \ref{rem: J-biinv}.

Using  Proposition \ref{Prop: Ugarte}  we   can suppose, without further restrictions,  that a  $6$-dimensional  nilpotent Lie algebra  $\frak g$  admitting a complex structure $J$ has complex structure equations \eqref{Eq: J non-nil}  if $J$  is non-nilpotent  and  \eqref{Eq: J nil1}  if   $J$ is nilpotent (with reduction \eqref{Eq: J nil0} if $\epsilon =0$).

A $1$-form $\a$ on  $\frak g$  can be written as
\begin{equation}\label{espralpha} \a = \a^{(1,0)} + \a^{(0,1)} = \sum_{j=1}^3 \lambda_j \omega^j + \sum_{j=1}^3 \overline \lambda_j \overline \omega^j,
\end{equation}
where  $\lambda_j \in \C$ and $\{\omega^j\}_{j=1}^3$ is the  basis of $(1.0)$-forms  satisfying the complex structure equations (\ref{Eq: J non-nil}) (resp. (\ref{Eq: J nil1}))  for $J$ non-nilpotent (resp.  $J$ nilpotent).

\begin{lemm}
Let $\a$ be a closed $1$-form on a 6-dimensional  nilpotent Lie algebra   $\frak g$  with complex structure $J$. If we write $\alpha$   with respect to the basis of $(1,0)$-forms as in  \eqref{espralpha}, then  we  get the following constraints on the coefficients $\lambda_j$:

\noindent (a) for $J$ non-nilpotent
\begin{equation}\label{Eq: a1}
    \lambda_2 = 0
\end{equation}
\begin{equation}\label{Eq: a2}
    \lambda_3 - \overline \lambda_3 E = 0
\end{equation}
\begin{equation}\label{Eq: a3}
\lambda_3 A - \overline \lambda_3 \overline A = 0
\end{equation}
(b) for $J$ nilpotent ($\epsilon = 0$)
\begin{equation}\label{Eq: b1}
    \rho \lambda_3 = 0
\end{equation}
\begin{equation}\label{Eq: b2}
    \lambda_3 = \overline \lambda_3
\end{equation}
\begin{equation}\label{Eq: b3}
  \lambda_3  {\rm Im} D= 0
\end{equation}
\begin{equation}\label{Eq: b4}
     \lambda_3 B = 0
\end{equation}
and for  $J$ nilpotent  ($\epsilon = 1$)
\begin{equation}\label{Eq: c1}
    \rho  \, \lambda_3 = 0
\end{equation}
\begin{equation}\label{Eq: c2}
    \lambda_2 = \overline \lambda_2
\end{equation}
\begin{equation}\label{Eq: c3}
\lambda_3 B - \overline \lambda_3 \overline C = 0
\end{equation}
\end{lemm}
\begin{proof} The constraints imposed by the condition $d\alpha =0$ are given by direct computation, using equations \eqref{Eq: J non-nil} in the non-nilpotent case, \eqref{Eq: J nil1} in the nilpotent case with $\epsilon =1$ and \eqref{Eq: J nil0}  in the nilpotent case with $\epsilon =0$.
\end{proof}
To investigate the existence of  $J$-Hermitian metrics satisfying the LCSKT condition  on  $(\frak g, J)$  we can use the fact that  the generic $J$-Hermitian metric $g$ on  $\frak g$ is expressed, in terms of the basis $\{ \omega^j \}$, as
 $$
 \begin{array}{lcl}
 g &=& r (\omega^1 \otimes \overline \omega^1 + \overline \omega^1 \otimes \omega^1) +  s (\omega^2 \otimes \overline \omega^2 + \overline \omega^2 \otimes \omega^2) + t  (\omega^3 \otimes \overline \omega^3 + \overline \omega^3 \otimes \omega^3) \\[3pt]
 &&- i u  (\omega^1 \otimes \overline \omega^2 + \overline \omega^2 \otimes \omega^1)  + i  \overline u (\omega^2 \otimes \overline \omega^1 + \overline \omega^1 \otimes \omega^2)\\[3pt]
&& - i v (\omega^2 \otimes \overline \omega^3 + \overline \omega^3 \otimes \omega^2)  + i \overline v (\omega^3 \otimes \overline \omega^2 + \overline \omega^2 \otimes \omega^3)\\[3pt]
&& - i z  (\omega^1 \otimes \overline \omega^3 + \overline \omega^3 \otimes \omega^1)  + i  \overline z (\omega^3 \otimes \overline \omega^1 + \overline \omega^1 \otimes \omega^3)
\end{array}
$$
where $r, s, t  \in  \R$ and $u, v, z  \in  \C$;  $r,s,t > 0, rs-|u|^2> 0, st-|v|^2 >0 , rt-|z|^2 > 0$,
$rst + 2 Re(i \overline u \overline v  z)- t |u|^2 - r |v|^2 - s|z|^2 > 0.$
These last conditions guarantee that the metric $g$ is positive definite, i.e., $g(Z,\overline Z)> 0$ for any nonzero $Z \in \frak g^{\C}$.\\
Furthermore, the fundamental form $\Omega$ of the  generic Hermitian structure $(J, g)$ is then given by
\begin{equation}\label{fundformgen}
\Omega =  i(r \omega^{1 \overline1}  + s  \omega^{2 \overline 2} + t \omega^{3 \overline 3}) + u \omega^{1 \overline 2} -  \overline u \omega^{2 \overline 1}
 + v \omega^{2 \overline 3} - \overline v   \omega^{3 \overline 2}+ z  \omega^{1 \overline 3}- \overline z  \omega^{3 \overline 1}
\end{equation}
For the 3-form $H=d^c\Omega$, we will be using the notation
\begin{equation*}
    H=H^{(2,1)}+H^{(1,2)} =\sum_{l<m, n=1}^{3} \left (H_{l m\overline n} \, \omega ^{lm \overline n}+ H_{n \overline l \overline m} \, \omega^{n \overline l \overline m}\right)\\
    = (-i\partial \Omega) + (i\overline \partial \Omega).
\end{equation*}
Notice also that $dH = dd^c\Omega = 2i \partial \overline \partial \Omega$.

The following result is proved by a direct calculation, so we omit the proof.
\begin{lemm}
Let $\Omega$ as in \eqref{fundformgen}. The $(2,1)-$part  $H^{(2,1)}$ of the 3-form  $H=d^c\Omega$ and its derivative $dH$ are given  respectively by:

(a) for $J$ non-nilpotent
$$
\begin{array}{cc}
H^{(2,1)}= i(\overline A v +ibz) \omega^{12 \overline 1}-bEv\omega^{12\overline 2}+i(i \overline At-u+E\overline u)\omega^{13\overline 1}-i(is+bt)E\omega^{13\overline 2}\\
-iEv\omega^{13\overline 3}-i(is-bt)\omega^{23\overline 1}
\end{array}
$$
\begin{equation}\label{dH0}
    dH  = -4 (b^2 t \omega ^{12\overline 1 \overline 2} +s\omega ^{13\overline 1 \overline 3})
\end{equation}
where $\{\omega^j\}_{j=1}^3$ is the basis of  $(1,0)$-forms satisfying \eqref{Eq: J non-nil}.\\
(b) For $J$ nilpotent $(\epsilon =0)$
$$
H^{(2,1)}=  i( \rho \overline z +v - \overline B z) \omega^{12 \overline 1}+i(\rho\overline v -\overline D z)\omega^{12\overline 2}+\rho t \omega^{12\overline 3}-t\omega^{13\overline 1}
-\overline B t \omega^{23\overline 1}-\overline D t\omega^{23\overline 2},
$$
\begin{equation}\label{dH1}
dH = -2t [\rho + |B|^2 -2 Re(D)]\omega^{12\overline 1 \overline2}
\end{equation}
where $\{\omega^j\}_{j=1}^3$ is the basis of $(1,0)$-forms satisfying (\ref{Eq: J nil0}).\\
For $J$ nilpotent $(\epsilon =1) $
$$
H^{(2,1)}=  i(is + \rho \overline z -\overline B z) \omega^{12 \overline 1}+i(\rho\overline v +\overline C v)\omega^{12\overline 2}+\rho t \omega^{12\overline 3}-i\overline v\omega^{13\overline 1}\\
-\overline C t \omega^{13\overline 2}-\overline B t\omega^{23\overline 1},
$$
\begin{equation}\label{dH2}
dH =  -2t [\rho + |B|^2 + |C|^2]\omega^{12\overline 1 \overline2}
\end{equation}
where $\{\omega^j\}_{j=1}^3$ is the basis of  $(1,0)$-forms  satisfying (\ref{Eq: J nil1}).
\end{lemm}

We will now classify 6-dimensional nilpotent Lie algebras $(\frak g,J)$, for which there exists a non-vanishing closed 1-form $(\a \neq 0)$ such that $dH=\a \wedge H$, where $H = d^c\Omega$ and $\Omega$ is given by (\ref{fundformgen}).   The case $\alpha =0$ was already studied in \cite{FPS}.
\begin{teo}\label{Thm: nilpotent-main}
Let $(\frak g,J)$ be a 6-dimensional nilpotent Lie algebra admitting a complex structure $J$. Then

\noindent  (a) If $J$ is non-nilpotent, then $(\frak g, J)$ does not admit any  LCSKT structure.

\noindent (b)  If $J$ is nilpotent, then $(\frak g, J)$ admits an LCSKT structure if and only if
it has either  complex structure equations
\begin{equation}\label{Eq: thm-nil-triv}
    \left\{  \begin{array}{l}
d \omega^j =0,  \quad j=1,2, \\ [3pt]
d \omega^3 = \omega^{1\overline 1},
\end{array}
\right.
\end{equation}
or
\begin{equation}\label{Eq: thm-nil-nontriv}
\left\{  \begin{array}{l}
d \omega^1 =0,\\
d \omega^2 = \omega^{1\overline 1},\\
d \omega^3 = \omega^{12} - \omega^{1 \overline 2}
\end{array}
\right.
\end{equation}
Moreover,    if $(\frak g, J)$ has  complex structure equations  \eqref{Eq: thm-nil-triv} and \eqref{Eq: thm-nil-nontriv}, every Hermitian structure is LCSKT.  In the  first case  every   LCSKT structure is trivial since $d H = 0$ and  the $1$-form $\a$ is independent of the parameters defining the Hermitian structure in (\ref{fundformgen})
 \begin{equation*}
    \a = 2Re(\lambda_1\omega^1),\quad \lambda_1 \in \C-\{0\}.
 \end{equation*}
In the second case every  LCSKT structure  is   non-trivial and  the $1$-form $\alpha$ is given by
\begin{equation*}
    \a = \frac{2it\overline v}{ts - |v|^2}\omega^1 -\frac{2itv}{ts - |v|^2} \overline \omega^1 - \frac{4t^2}{ts - |v|^2} Re(\omega^2)
\end{equation*}
and  by \cite{FPS} the nilpotent  Lie algebra with   complex structure equations \eqref{Eq: thm-nil-nontriv} does not admit any SKT structure.
\end{teo}

\begin{proof}
Firstly,  we can rewrite  the equation $dH =\a \wedge H$ as
\begin{equation}\label{Eq: LCSKT-condition}
    dH -\a \wedge H = 0
\end{equation}
In the  proof  we will study separately the two cases:  $J$ non-nilpotent and $J$ nilpotent.
\smallskip

\noindent (a) For $J$ non-nilpotent: by imposing the vanishing of the coefficient of $\omega^{23\overline 1 \overline 3}$ in \eqref{Eq: LCSKT-condition} and by using the constraints (\ref{Eq: a1}) (i.e $\lambda_2= 0$) and (\ref{Eq: a2}), we obtain
$\overline \lambda_3(-EH_{2\overline 1 \overline 3}-H_{23\overline 1})=0,$ i.e.  $\overline \lambda_3 (is-bt)=0,$
which implies $\lambda_3=0$ since $s,t > 0$ and  $b\in \mathbb{R} \backslash\{ 0\}$.\\
Then, by the vanishing of the coefficient of $\omega^{12\overline 1\overline 3}$ in \eqref{Eq: LCSKT-condition}, we get $\lambda_1 H_{2\overline 1 \overline 3} = 0 $,
which implies $\lambda_1 = 0$, since $H_{2\overline 1 \overline 3} = i(-is+bt)\overline E \neq 0$.\\
Therefore, if $J$ is non-nilpotent,  then  $\a = 0$ and $(\frak g,J)$ cannot admit any non-trivial  LCSKT neither any SKT structure, since $dH \neq 0$ from \eqref{dH0}.

\noindent (b) For $J$ nilpotent, we consider  separately the two cases: (b1)  $\epsilon =0$ and  (b2)  $\epsilon =1$.\\
  \begin{enumerate}
\item[(b1)] From the constraint (\ref{Eq: b2}), i.e.  $\lambda_3 \in \R$,  and  by the vanishing of the coefficient of $\omega^{13\overline 1 \overline 3}$ in \eqref{Eq: LCSKT-condition}, we obtain  $2t\lambda_3=0$ so $\lambda_3=0$  $(t \neq 0)$.\\
By the vanishing of the coefficient of $\omega^{123\overline 2}$ and of $\omega^{123\overline 1}$ in \eqref{Eq: LCSKT-condition},
we get $\lambda_1 \overline D =0 $ and $\lambda_2 = \overline B \lambda_1.$
If $D \neq 0$, we have $\lambda_1 = 0$ and $\lambda_2 = 0,$ so $\a = 0.$
If $D=0$, we distinguish the two cases $\rho=0$ (i.e.  $J$ abelian) and $\rho=1$ (i.e. $J$ non-abelian).\\
If $\rho = 0$, by the vanishing of the coefficient of $\omega^{12\overline 1 \overline 3}$ in \eqref{Eq: LCSKT-condition}, we obtain $t\lambda_2=0$, so $\lambda_2=0$. By the vanishing of the coefficient of $\omega^{12 \overline 1 \overline 2}$ in \eqref{Eq: LCSKT-condition}, we get $-2t|B|^2=0$,  so for $B = 0$ with an arbitrary $\lambda_1 \neq 0$, we can deduce that every $\Omega$ defines a locally conformal SKT structure on the Lie algebra with complex structure equations
$$
\left \{ \begin{array}{l}
d \omega^j =0,\, j=1,2\\[3pt]
d \omega^3 =  \omega^{1\overline 1},
\end{array}
\right.
$$
with a non-unique $\a= 2Re(\lambda_1\omega_1),$ $\lambda_1 \neq 0.$ Remark that $H$ is degenerate with $dH=0$.\\
If $\rho=1$, by the vanishing of the coefficient of $\omega^{12\overline 1\overline 3}$ and of $\omega^{12\overline 2\overline 3}$ in \eqref{Eq: LCSKT-condition}, we get
 $\overline \lambda_1 = -\lambda_2$ and $\overline \lambda_2=-B\lambda_2,$ which implies that $\lambda_2(|B|^2-1)=0$.\\
If $|B|\neq 1$, we have $\lambda_2=0$ and $\lambda_1=0$, so $\alpha=0.$\\
If $|B|=1,$ by the vanishing of the coefficient of $\omega^{12\overline 1\overline 2}$ in \eqref{Eq: LCSKT-condition} we get $-4t=0$ which is impossible $(t \neq 0)$, so in this case $(\frak g,J)$ cannot admit any locally conformal SKT structure.\\

    \item [(b2)]  Since    $\epsilon=1$, by the constraint \eqref{Eq: c2} we get $\lambda_2 \in \R$.  We distinguish two cases:\\
 i)  $\rho=0$, i.e. $J$ abelian  and  ii) $\rho=1$, i.e $J$ non-abelian.\\
In the case i),  if we  have $B= C=0$,  from \eqref{dH2} we get  $dH = 0$.\\
By the vanishing of the coefficients of $\omega^{12\overline 1 \overline 2}$ and of $\omega^{12\overline 1 \overline 3}$ in \eqref{Eq: LCSKT-condition}, we get $\lambda_2 s=0$ and $\lambda_3=i \frac{\overline v}{s} \lambda_2,$ so $\lambda_2=\lambda_3=0$.
As a consequence, for arbitrary $\lambda_1 \neq 0$  we have  a trivial  LCSKT  structure. Also, notice that and after interchanging $\omega^2$ with $\omega^3$ we obtain again the same structure equations (\ref{Eq: thm-nil-triv}).\\
If $B = 0$ and $C \neq 0$, by (\ref{Eq: c3}) we get $\lambda_3=0$. By the vanishing of the coefficients of $\omega^{123\overline 3}$ and  of $\omega^{12\overline 1 \overline 3}$ in \eqref{Eq: LCSKT-condition} we get $\lambda_2 t\overline C = 0$, $Ct\lambda_1  + i v\lambda_2  = 0,$ so $\lambda_2=0$ and $\lambda_1=0$, i.e. $\a=0$.\\
If $C = 0$ and $B \neq 0,$ by (\ref{Eq: c3}) we obtain $\lambda_3=0$. By the vanishing of the coefficients of $\omega^{12\overline 2\overline 3}$ and of  $\omega^{123\overline 1}$ in \eqref{Eq: LCSKT-condition} we get $\lambda_2 tB= 0$, $\lambda_1 \overline B=0$ so $\lambda_2=0$ and $\lambda_1=0$, i.e $\a=0$.\\
If $C \neq 0$ and $B \neq 0,$  by the vanishing of the coefficients of  $\omega^{123 \overline 2}$,  $\omega^{13\overline 1 \overline 3}$ and of
 $\omega^{123\overline 1}$  in \eqref{Eq: LCSKT-condition}, we get the following system
$$
\left \{  \begin{array}{l}
   \lambda_2 = - i\frac{v}{t} \lambda_3 ,\\[3pt]
 \overline \lambda_3 \overline v = \lambda_3 v,\\[3pt]
  -\overline B t \lambda_1 + i \overline v \lambda_2 + i(is-\overline B z)\lambda_3 = 0
 \end{array}
 \right.
 $$
By using the first and second equations, we obtain
$$\overline \lambda_2 =-i \frac{\overline v}{t} \overline \lambda_3 =- i(\frac{v}{t}\lambda_3) = -\lambda_2$$
since $\lambda_2= \overline \lambda_2$ (constraint (\ref{Eq: c2})), which implies that  $\lambda_2=0$, so either  $\lambda_3=0$ or $v=0$.\\
If $\lambda_3=0$, the third equation implies that  $-\overline Bt\lambda_1=0$ so  $\lambda_1=0$, i.e. $\a=0$.\\
If $v=0$, by the vanishing  of the coefficients of  $\omega^{12\overline 1 \overline 2}$  in \eqref{Eq: LCSKT-condition}, we obtain $2t [|B|^2+|C|^2 ]=0$ so $B=C=0$ and we get a contradiction.\\

In the case ii), i.e $(\rho=1)$, then from the constraint (\ref{Eq: c1}) we have $\lambda_3=0$. By the vanishing of the coefficients of $\omega^{123\overline 2}$,  $\omega^{12\overline 2 \overline 3}$, $\omega^{12\overline 1 \overline 3}$ and of $\omega^{123\overline 1}$ in \eqref{Eq: LCSKT-condition},\\
 we obtain the following system
$$
\left \{  \begin{array}{l}
   \lambda_2 t \overline C =0,\\[3pt]
   \lambda_2t (B+1)=0,\\[3pt]
 -  \lambda_1 t C - i \lambda_2 v + \overline \lambda_1 t =0,\\
  -\lambda_1 t \overline B + i \lambda_2 \overline v = 0
 \end{array}
 \right.
 $$
In the first case $C \neq 0$, by the first equation we get $\lambda_2=0$ and if $B\neq 0$ by the fourth equation we have $\lambda_1=0$, so $\a=0$.
If $B=0$ and $\lambda_1 \neq 0$, from the third equation we get $\overline \lambda_1=C\lambda_1,$ which implies that $|C|^2=1$. Moreover, by the vanishing of the coefficient of $\omega^{12\overline 1 \overline 2}$ in \eqref{Eq: LCSKT-condition} we obtain $-4t=0$ and we get a contradiction $(t\neq 0)$.\\
In the second case $C=0$, for $B \neq -1$, by the second equation we get $\lambda_2=0$  and the third equation implies that $\lambda_1=0$, so $\a=0$.\\
If $C=0$ and $B=-1$, by the third and the fourth equations we obtain
$$\lambda_1 = -i\frac{\overline v}{t} \lambda_2, \quad \overline \lambda_1= i\frac{v}{t} \lambda_2$$
By the vanishing of the coefficient of $\omega^{12\overline 1 \overline 2}$ in \eqref{Eq: LCSKT-condition} and by substituting the last expressions of $\lambda_1$ and $\overline \lambda_1$, we obtain
$$-4t -2 \left (\frac{ts - |v|^2}{t} \right )\lambda_2 = 0 $$
so $\lambda_2 = \overline \lambda_2 = \frac{-2t^2}{ts - |v|^2} \neq 0$ and $\lambda_1 = \frac{2it\overline v}{ts - |v|^2}$, since $ts - |v|^2\neq 0$.\\

Hence, every 2-form $\Omega$, given by \eqref{fundformgen},  defines a  non-trivial LCSKT structure on the nilpotent Lie algebra $(\frak g,J)$ with the complex structure equations
$$
\left \{  \begin{array}{l}
 d \omega^1 =0,\\[3pt]
 d \omega^2 =\omega^{1\overline 1},\\[3pt]
 d \omega^3 = \omega^{12} - \omega^{1\overline 2},
 \end{array}
 \right.
 $$
with a unique (non-zero) closed $1$-form
$$
\a = \frac{2it\overline v}{ts-|v|^2}\omega^1 - \frac{2itv}{ts-|v|^2}\overline \omega^1 -\frac{4t^2}{ts-|v|^2} Re(\omega^2).
$$
\end{enumerate}

\end{proof}

\begin{remark}\label{rem: J-biinv}
If $\frak g$ is not abelian and  $J$ is bi-invariant, then $(g,J)$  has structure equations
$$\left\{\begin{array}{l}
d\omega^1 = d\omega^2 =0,\\
d\omega^3 = \omega^{12}.
\end{array} \right.$$
Analogous computations to those of Theorem \ref{Thm: nilpotent-main} show that for a fundamental form as in  eq.  \eqref{fundformgen},  we get $dH = 2t \omega^{12\overline{12}} $.  If $\alpha$ as in eq. \eqref{espralpha} is a closed 1-form then $\lambda_3 =0$. Imposing the condition $dH - \alpha \wedge H =0 $ implies, by the vanishing of the coefficients $\omega_{13\overline{12}}$ and $\omega_{23\overline{12}}$,  that $\lambda_1 = \lambda_2 = 0$.  Thus $\alpha =0$, which is a contradiction.
Then,  if $J$ is bi-invariant, there does not exist any $J$-Hermitian metric $g$ such that $(J,g)$ is LCSKT.
\end{remark}

In order to determine, up to isomorphism, the  underlying   $6$-dimensional LCSKT nilpotent Lie algebras,   we will adopt for the list of Lie algebras  the notation ${\frak h}_k$ used in \cite[Theorem 8]{CFGU}.    Moreover,  for instance by $(0,0,0,0,0,12)$  we  will denote the nilpotent  Lie algebra with real structure equations

$$\left \{
\begin{array}{l}
de^k = 0,\, k=1,...,5, \\
 de^6 = e^1 \wedge e^2,
 \end{array}
 \right.$$
where  $\{ e^j \}$ is a basis of real $1$-forms of the Lie algebra and $d$ denotes  the Chevalley-Eilenberg differential.

\begin{corol}\label{Cor: nilpotent-main}
A 6-dimensional nilpotent Lie algebra  $\frak g$ admits an LCSKT structure $(J, g)$  if and only if $\frak g$ is isomorphic either to   ${\frak h}_8 = (0,0,0,0,0,12)$ or ${\frak h}_{16}= (0,0,0,12,14,24)$.  Moreover,  if $\frak g \cong \frak h_8$, the underlying complex structure $J$  must be abelian and every LCSKT structure is trivial. If $\frak g \cong \frak h_{16}$, $J$  has to be (non-abelian) nilpotent and every LCSKT structure is  non-trivial.
\end{corol}

\begin{proof}
Let $\{f^j\}_{j=1}^6$ be the basis   of  real 1-forms  on $\frak g$ such that
\begin{equation}\label{Eq: real-forms}
   \omega^1=f^1+i f^2, \, \, \omega^2= f^3+if^4, \, \, \omega^3=f^5+if^6.
\end{equation}
To obtain the real structure equations  of $\frak g$ we  use \eqref{Eq: real-forms} and we   impose   \eqref{Eq: thm-nil-triv} and \eqref{Eq: thm-nil-nontriv}, respectively.  For case (\ref{Eq: thm-nil-triv}),
if we consider the change of basis
\begin{equation*}
    e^1= -\sqrt{2}f^2, \, e^2 = \sqrt{2}f^1, \, \, e^k= f^k,\,\, k=3, \ldots,6,
\end{equation*}
we obtain the structure equations  of ${\frak h}_8=(0,0,0,0,0,12)$.
For case (\ref{Eq: thm-nil-nontriv}),
considering the change of basis
\begin{equation*}
    e^1=  \sqrt{2}f^2, \, e^2 = \sqrt{2}f^1, \, \, e^3= f^3,\,\, e^4= f^4\,\,
    e^5= -\frac{1}{\sqrt{2}} f^5, \, e^6 = \frac{1}{\sqrt{2}}f^6,
\end{equation*}
we obtain the structure equations
of  the 3-step  nilpotent Lie algebra ${\frak h}_{16}=(0,0,0,12,14,24)$.  On $\frak h_8$, for  every LCSKT structure we have
$d H =0$ and the LCSKT structure is trivial.  On the other hand,  on $\frak h_{16}$ we have always $d H \neq 0$ and a non-trivial LCSKT structure.
\end{proof}

\noindent\begin{remark} \label{Rem: nilpotent}
Note that  $\frak h_{16}$  cannot admit any SKT structure since it is $3$-step nilpotent, \cite{FPS, AN}. Moreover, by \cite[Prop. 25 and Th. 31]{Ug} $\frak h_8$ and ${\frak h}_{16}$ do not admit any balanced (or LCK)  Hermitian structure.
The simply connected Lie groups with Lie algebras $\frak h_8$, $\frak h_{16}$ both admit  lattices,
since every  simply connected nilpotent Lie  group  whose Lie algebra    has   rational  structure constants  has a lattice (see \cite[Theorem 2.1]{MSR} and \cite{AM45}).
\end{remark}

\noindent\begin{remark}
Nilmanifolds of dimension 8 admitting balanced metrics have recently been constructed in \cite{LUV}. It would be a natural question to investigate if such manifolds can simultaneously satisfy the LCSKT condition.
\end{remark}

As a consequence of Theorem \ref{Thm: nilpotent-main}, Corollary \ref{Cor: nilpotent-main}  and Remark  \ref{Rem: nilpotent} we have the following

\begin{teo} If $M^6 =\Gamma \backslash G$ is a  $6$-dimensional nilmanifold, then $M^6$ admits an invariant LCSKT structure $(J, g)$  if and only if  the Lie algebra of $\frak g$ of $G$  is isomorphic either to $\frak h_8$  or $\frak h_{16}$.  The LCSKT  structure is non-trivial (resp. trivial)   if and only if $\frak g$ is isomorphic to $\frak h_{16}$ (resp. $\frak h _{8}$). Moreover, $M^6$ does not admit any invariant balanced or invariant  LCK structure.  If $\frak g \cong \frak h_{16}$, $M^6$ cannot have any SKT structure.

\end{teo}

\begin{remark}
Note that according to \cite[Proposition 6.1] {FPS}, the holonomy group  of the Bismut connection $\nabla^B$ of a  6-dimensional nilmanifold $M^6= \Gamma/G$ with an invariant   LCSKT structure  $(J,g)$ is not reduced to $SU(3)$, since the LCSKT  structure  $(J,g)$ is not balanced.
\end{remark}

Using  the complex structure equations \eqref{Eq: thm-nil-triv} and \eqref{Eq: thm-nil-nontriv} associated to  $(M^6, J, g)$  (see Theorem \ref{Thm: nilpotent-main}),  one can check that  the  Hermitian metric with fundamental form
\begin{equation*}
    \Omega  = \frac{i}{2} (\omega^{1 \overline 1} + \omega^{2 \overline 2} + \omega^{3\overline 3})
\end{equation*}
is  LCB with $\rho^B \neq 0 $.
Indeed,  if we  choose the  basis $\{f^j\}$ of $1$-forms  such that  $$\omega^1 = f^1 + i f^2, \, \omega^2 = f^3 + i f^4, \, \omega^3 = f^5 + i f^6,
$$
we have $d\Omega =2 f^{125}$ and   $\theta =f^5$ for $\frak g \cong \frak h_8$   and
for $\frak g\cong \frak h_{16},$  we obtain $d\Omega = f^{123} - f^{145} -f^{246}$ and $\theta =f^3.$  As a consequence in both cases $d\theta=0$. To compute the Bismut Ricci form we can use that  $\rho^C$ vanishes by \cite{LRV}.
 Therefore, as a consequence of \eqref{Eq: Ricci-forms}\begin{equation*}
    \rho^B =-2 dJ\theta
\end{equation*}
and we get for both  cases  $\rho^B= 4 f^{12}$.

\section{Invariant LCSKT structures on almost abelian solvmanifolds}\label{Sec: almost-abelian}

In this section we study the existence of invariant  LCSKT structures on almost abelian  solvmanifolds,  i.e. on     compact  quotients $M = \Gamma/G$, where $G$ is a connected and simply connected almost abelian   Lie group and $\Gamma$ is a lattice in $G$.  Therefore, we will determine the necessary and sufficient conditions for the existence of an LCSKT structure on   an almost abelian Lie algebra $\frak g$.
We recall that a  Lie algebra $\frak g$ is called {\em almost abelian}  if it contains a codimension-one abelian ideal $\frak n,$ so it is the next simplest  type of Lie algebra after abelian Lie algebras.  Let  $J$ be a complex structure on $\frak g$, then $\frak n^1 :=  \frak n \cap  J \frak n$  is the maximal $J$-invariant subspace of $\frak n$.
By  \cite{AL, LRV}, if $(J,g)$ is a Hermitian structure on $\frak g$, then  there exists a unitary  basis $\{e_1,\ldots ,e_{2n}\}$ of  $(\frak g, J, g)$ such that $$\frak n = span \langle e_1, \ldots,e_{2n-1} \rangle,  \frak n_1= span \langle e_2, \ldots,e_{2n -1} \rangle, J e_i = e_{2n+1-i},  \, i = 1,.\ldots, n.$$
So in particular
\begin{equation*}
    Je_1 = e_{2n}, \quad J(\frak n_1) \subset \frak n_1.
\end{equation*}
We set $J_1:= J\vert_{\frak n_1}$ and the  basis  $\{e_j\}_{j=1}^{2n}$ will be called  \emph{adapted} to the  orthogonal decomposition (splitting)
$$
\frak g = \langle {e_1}  \rangle \oplus \frak n_1  \oplus  \langle e_{2n} \rangle, \, \frak n = \langle {e_1} \rangle  \oplus \frak n_1.
$$
 The matrix  $B$  associated to $(ad_{e_{2n}})\vert_{\frak n}$ determines the whole information about the Lie algebra structure  and $\frak g$  is therefore isomorphic to the semidirect product $\R \ltimes ad_{e_{2n}}\vert_{\frak n}  \R^{2n-1}$.

As proved  in  \cite{LRV},  the matrix $B$ associated with $ad_{e_{2n}}\vert_{\frak n }$  in  the adapted basis  $\{e_j\}_{j=1}^{2n}$ is of the form
$$
B := (ad_{e_{2n}})\vert_{\frak n } = \left(
  \begin{array}{cc}
    a & 0 \\
    v & A \\
  \end{array}
\right),  \quad a \in \R, \, v \in \frak n_1, \, A \in {\frak {gl}}  (\frak n_1), \, \, [A, J_1]=0.
$$
Here, we are identifying $v\in\frak n_1$ with its coordinates in the basis $\{e_2, \cdots, e_{2n-1}\}$ freely.
In terms of commutation relations, this implies
\begin{equation}\label{Eq: alm-ab-brackets}
    [e_{2n},e_{1}] = ae_1+ v,\quad [e_{2n},X]= AX, \quad X \in \frak n_1
\end{equation}

In this way, every  Hermitian almost abelian Lie algebra $(\frak g, J,g)$ is characterized by the triplet $(a,v,A)$ or, equivalently, by the real $(2n-1)\times (2n-1)$ matrix corresponding to $B$.  We  will denote such an algebra by $\frak g(a,v,A)$.
Let us recall the conditions imposed on the algebraic data $(a,v,A) $, for which an almost abelian Lie algebra admits a K\"{a}hler or an SKT structure.

\begin{teo}[\cite{LW}]
A    Hermitian almost abelian Lie algebra  $(\frak g (a,v,A), J, g)$  of dimension $2n$  is  K\"{a}hler  if and only if
\begin{equation*}
A \in \frak{so} (\frak n_1) \quad and \quad v=0
\end{equation*}
\end{teo}

\begin{remark} \label{LCBalmostab} Notice that,  by Theorem 2.2 in \cite{P},  a   Hermitian almost abelian Lie algebra  $(\frak g (a,v,A), J, g)$  is LCB if and only if  $A^tv = 0.$ \end{remark}

\begin{teo}[\cite{AL}]
A $2n$-dimensional  Hermitian almost abelian Lie algebra  $(\frak g (a,v,A), J, g)$ is  SKT  if and only if $(aA+A^2+A^t A) \in \frak {so}(\frak n_1)$
 or, equivalently,  if $A$ is normal (i.e $[A,A^t]=[A,J_1]=0$) and each eigenvalue of $A$ has real part equal to $0$ or $-\frac{a}{2}$.
\end{teo}
Now, according to  formula   (3.2) in \cite{EFV} (see also  \cite{DF})  for every Lie algebra  $\frak g$ endowed with a Hermitian structure $(J, g)$ the Bismut  torsion 3-form   $H$ is given by
\begin{equation}\label{Eq: H-lie-alg}
 H(X,Y,Z)= -g([JX,JY],Z) -g([JY,JZ],X)- g([JZ,JX],Y).
\end{equation}
and the exterior derivative $dH$ can be computed by the usual formula
\begin{eqnarray}
  dH(W,X,Y,Z) &=& -H([W,X],Y,Z) + H([W,Y],X,Z) - H([W,Z],X,Y) \nonumber \label{dH} \\
   && - H([X,Y],W,Z) + H([X,Z],W,Y) - H([Y,Z],W,X),
\end{eqnarray}
for any vector $X,Y,Z,W \in \frak g$.

\begin{prop}\label{Prop: alm-ab-conditions}
An  almost abelian Lie algebra $\frak g:= \frak g (a,v,A)$ with Hermitian structure $(J, g)$  is LCSKT if and only if there exists a non-zero  closed-$1$-form $\a \in \frak g^*$ such that the following conditions are satisfied
\begin{equation}\label{c1}
    a\a(e_1) + \a(v)= 0, \quad A^t \a\vert_{\frak n_1}=0
\end{equation}
\begin{equation}\label{c2}
    \a(X)g(S(A)J_1Y,Z)- \a(Y)g(S(A)J_1X,Z)+\a(Z)g(S(A)J_1Y,X)= 0
\end{equation}
\begin{equation}\label{c3}
  g(S((a+ \a(e_{2n}))A + A^2 + A^t A)J_1Y,Z) = \frac{1}{2}( g(v,Y)\a(Z) - g(v,Z)\a(Y)),
\end{equation}
for every $X, Y, Z \in \frak n_1$, and where $S(A)=\frac{1}{2}(A+A^t)$ is the symmetric part of $A$.
\end{prop}
\begin{proof}
For every $1$-form $\a \in \frak g^*,$ we have
$$d\a (Y,e_{2n}) = - \a ([Y,e_{2n}]) = \a(ad_{e_{2n}}(Y)),  \quad \forall Y \in \frak g.$$
Using   the commutation relations \eqref{Eq: alm-ab-brackets} we get
$$
d\a (e_1,e_{2n}) = \a (ad_{e_{2n}}(e_1))= a \a(e_1)+ \a(v) $$
and $$d\a(X,e_{2n}) = \a (ad_{e_{2n}}(X)) = \a(AX), \quad \forall X \in \frak n_1.$$  By imposing $d\a = 0,$ it  follows that   $a\a(e_1)+\a(v)=0$ and $\a(AX)=0$, for every $X\in \frak n_1$, i.e $A^t\a \vert _{\frak n_1}= 0$.  Therefore, we obtain the condition \eqref{c1}.

Taking into account that $\frak n_1$ is an abelian ideal preserved by $J$ and by   using \eqref{Eq: H-lie-alg} we clearly get
$H(X,Y,Z)= 0,$  whenever $X,Y,Z \in \frak n_1.$ Then, by  \eqref{dH}, one quickly checks that $dH(W,X,Y,Z) = 0,$ if three of the entries lie in $\frak n_1.$\\
Therefore, by imposing
$dH(e_1,X,Y,Z) = (\a \wedge H)(e_1,X,Y,Z)$ we get
\begin{equation*}
     -\a (X) H(e_1,Y,Z) +\a(Y)H(e_1,X,Z) - \a(Z)H(e_1,X,Y) = 0
\end{equation*}
Now, if  we  compute the expression of $ H(e_1, Y, Z),$ for any  $Y,Z \in \frak n_1$, we get
\begin{eqnarray}
 H(e_1, Y, Z) &=& -g([Je_1, JY],Z) -g([JZ,Je_1],Y) \nonumber \\
   &=& -g([e_{2n}, J_1Y],Z)  + g([e_{2n},J_1Z],Y)  \nonumber \\
   &=& -g(A JY,Z) + g(AJZ,Y) \nonumber \\
   &=& -g(A J_1Y,Z) - g(J_1A^tY,Z) = -2 g(S(A)J_1Y,Z) \label{Eq: H(ei,Y,Z)}
\end{eqnarray}
In the last step of the computation above,  we  used   that $J_1$ commutes with $A^t$  since $J_1$ commutes with $A$ and $J_1^t=-J_1$.
Applying \eqref{Eq: H(ei,Y,Z)} in all summands of \eqref{Eq: alm-ab-brackets}, we obtain   the condition \eqref{c2}.
Moreover, if we use the formula \eqref{Eq: H-lie-alg} and the commutation relations \eqref{Eq: alm-ab-brackets} again, we get
$$H(e_{2n},Y,Z)= 0, \quad H(e_{2n}, e_1, Z) = -g(v,Z), \quad \forall Y,Z \in \frak n_1. $$
Similar computations show that
$$dH(e_{2n}, e_1, Y, Z) = 2g(S(aA + A^2 + A^t A)JY,Z)$$
and
\begin{eqnarray*}
 (\a \wedge H)(e_{2n},e_1,Y,Z) &=& \a(e_{2n})H(e_1,Y,Z)-\a(e_1)H(e_{2n},Y,Z) \\
  &+&  \a(Y)H(e_{2n},e_1,Z) - \a(Z)H(e_{2n},e_1,Y)  \\
   &=& -2 \a(e_{2n})g(S(A)J_1Y,Z) - \a(Y)g(v,Z) + \a(Z)g(v,Y).
\end{eqnarray*}
By imposing $$dH(e_{2n},e_1, Y,Z) = (\a \wedge H)(e_{2n},e_1, Y,Z)$$
we get the condition \eqref{c3}\\

 $g(S((a+\a(e_{2n}))A + A^2 + A^tA)J_1Y,Z) = \frac{1}{2}(g(v,Y)\a(Z) - g(v,Z)\a(Y)).$
\end{proof}
With respect the dual basis   $\{e^i \}_{i=1}^{2n}$  of the  adapted   basis  $\{e_i\}_{i=1}^{2n}$ of $\frak g$   we can write the $1$-form  $\a \in \frak g^*$ as
\begin{equation*}
    \a = \sum_{i=1}^{2n}\lambda_i e^i =  \lambda_1 e^1 +\sum_{k=2}^{2n-1}\lambda_ke^k + \lambda_{2n}e^{2n},
\end{equation*}
where $\lambda_i = \a (e_i).$\\

\begin{remark} Note  that the  expressions  of the Chern- and Bismut-Ricci forms   were  obtained respectively in \cite[Lemma 6.1]{LRV} and \cite[Prop. 4.8]{AL} (see also \cite{FP} for a correction in the expression of $\rho^B$)  and are given by
\begin{equation*}
    \rho^C = -(a^2 + \frac{1}{2}a \,  Tr A)e^1 \wedge e^{2n},
\end{equation*}
\begin{equation*}
    \rho^B = -(a^2 - \frac{1}{2}a \, TrA + \|v\|^2)e^1 \wedge e^{2n} - (A^t v)^{\flat} \wedge e^{2n},
\end{equation*}
where $X^{\flat} (.):= g(X,.)$, for $X \in \frak g$.  Therefore, by  \cite{FP20} it follows  that  $\rho^C = \rho^B$ if  and only if
$$
(a \, TrA -|v|^2)e^1 \wedge e^{2n}  - (A^t v)^\flat \wedge e^{2n} =0.
$$
\end{remark}

We can prove the following result.
\begin{prop}
 Let $(J,g)$ be a Hermitian structure on a $2n$-dimensional almost abelian Lie  algebra $\frak g$.  If $(J,g)$ is balanced and LCSKT, then it is necessarily K\"{a}hler.
\end{prop}
\begin{proof}
The Lee form $\theta$ for $2n$-dimensional almost abelian Lie algebra $\frak g$, is given by
\begin{equation}\label{Eq: Lee}
    \theta = -(tr A) e^{2n} + (Jv)^{\flat}
\end{equation}
(see  \cite[Lemma 2.1]{FP20}).
By \cite[Th. 3.1]{FP20} we get  that $(J, g)$ is  balanced  $(\theta=0)$ if and only if $v=0$ and $tr A=0$. If $(J,g)$ is also LCSKT,  the  condition \eqref{c3} reduces to $$S((a+\lambda_{2n}))A+A^2+A^t A))=0.$$
Taking traces in this last equation,
 yields to $(a+\lambda_{2n})tr A + 2tr S(A)^2=0$ so $tr S(A)^2=0$, which implies that $S(A)=0$. In this case $v=0,$ $S(A)=0$ the Hermitian structure $(J,g)$  has to be  K\"{a}hler (\cite{LW}).
\end{proof}

We will now study the  equations  in  Proposition \ref{Prop: alm-ab-conditions} by  imposing  the additional assumption  that $A$  is non-degenerate, i.e. $detA \neq 0$. We will see that we obtain similiar results to those in \cite{AL} for SKT structures.

\begin{lemm}
If $ \frak g(\mu(a,v,A),J,g)$ is a $2n$-dimensional almost abelian  Hermitian  Lie algebra with $det A \neq 0$, then the Hermitian structure  $(J,g)$ is LCSKT with closed 1-form $\alpha$ if and only if
\begin{equation}\label{Eq: alm-ab-alpha}
    \a \vert_{\frak n_1} = 0, \quad a\a(e_1)=0
\end{equation}
and
\begin{equation}\label{Eq: alm-ab-matrix}
    ((a+\a(e_{2n}))A+A^2+A^t A) \in \mathfrak{so}(\frak n_1).
\end{equation}
\end{lemm}
\begin{proof}
If $det A \neq 0$ by the second condition of \eqref{c1} we obtain $\a \vert_{\frak n_1} = 0,$ and thus the first condition becomes $a \a (e_1) = 0.$
Consequently, \eqref{c2} is trivially satisfied and \eqref{c3} is reduced to
$g(S((a+\a(e_{2n}))A+A^2+A^t A))J_1Y,Z)=0$ for any $Y,Z \in \frak n_1$
which in turn implies that
$S((a+\a(e_{2n}))A+A^2+A^t A))=0$ and so $((a+\a(e_{2n} ))A+A^2+A^t A)\in \frak so(\frak n_1).$
\end{proof}

\begin{lemm}\label{Lem: alm-ab-LCSKT}
Let $\frak g(\mu(a,v,A),J,g)$ be an almost abelian Hermitian Lie algebra such that $detA \neq 0$.  If  $(J,g)$ is LCSKT  with closed 1-form $\alpha$, then the real part of the eigenvalues of $A$ are either
$0$ or $-\frac{1}{2}(a+\lambda_{2n}),$
where $\lambda_{2n}=\a(e_{2n}).$
\end{lemm}
\begin{proof}
We will proceed in a similar way as  in  the proof of \cite[Lemma 4.8]{AL}.

Consider the complexification $\frak n_1^\C$ of $\frak n_1$ and extend $A$ linearly to $\frak n_1^\C$.
Let $k \in \C$ be an eigenvalue of $A$ corresponding to the eigenvector
 $z=x+\sqrt{-1}y =x+iy \in \frak n_1^{\C},x,y \in \frak n_1.$ Since  $((a+\lambda_{2n})A+A^2+A^t A) \in \mathfrak{so}(\frak n_1)$, we have
$$\langle\langle((a+\lambda_{2n})A + A^2+A^t A)z,\overline{z}\rangle\rangle = 0 =((a+\lambda_{2n})k+2k^2)\langle\langle z,\overline{z}\rangle\rangle =k(a+\lambda_{2n} +2k)\langle\langle z,\overline{z}\rangle\rangle, $$
where $\langle\langle.,.\rangle\rangle$ is the Hermitian inner product on $\frak n_1^{\C}$ induced by a fixed inner product  $g=\langle.,.\rangle$ on $\frak n_1$.
We have two cases to consider. Either
 $k=0$,$ -\frac{1}{2}(a+\lambda_{2n})$ (and the lemma follows) or
 $k \neq 0,$ $-\frac{1}{2}(a+\lambda_{2n}),$ and in this case $\langle\langle z, \overline{z}\rangle\rangle = 0$ which implies that $\|x\|= \|y\|$ and $\langle x, y\rangle =0.$

Writing $k = m + \sqrt{-1} n = m+in$, with $m,n\in \R$, and using that $Ax = mx - ny$ and $Ay = nx + my,$ we get
\begin{eqnarray*}
 0 &= & \langle ((a+\lambda_{2n})A+A^2+A^tA)u,u\rangle\\ &=&  [ (a+\lambda_{2n})m + m^{2}-n^{2} + m^{2} + n^{2})] \|x\|^2 \\
   &=& m(a+\lambda_{2n} + 2m) \|x\|^2.
\end{eqnarray*}
Then we can conclude that $m = 0$ or $m = -\frac{1}{2}(a+\lambda_{2n})$ from which the lemma follows.

\end{proof}
In order to establish  our main result in this section, we need to recall the following linear algebra estimate,
\cite[Corollary A2]{AL}. Consider the norm in $gl_m(\R)$, defined by the equality $\| B\|^2 = tr (BB^t).$

\begin{lemm}\label{Lem: normalop-estimate}
For any $E \in gl_m (\R)$ with eigenvalues $k_1,...,k_m \in \C$, we have that
$$\|S(E)\|^2 \geq \sum_{i=1}^m Re(k_i)^2$$
with equality if and only if $E$ is normal.  \end{lemm}

\begin{teo}\label{Thm: alm-ab-main}
Let $\frak g:=  \frak g(\mu(a,v,A),J,g)$  be  a Hermitian almost abelian Lie algebra with $det A \neq 0.$ The Hermitian  structure $(J, g)$  is LCSKT with closed 1-form $\alpha$ if and only if
$[A,A^t] = [A,J_1] = 0$ and each eigenvalue of $A$ has a real part equal to $0$ or $ -\frac{1}{2}(a+\lambda_{2n})$, where $\lambda_{2n}=\alpha(e_{2n})$.
\end{teo}
\begin{proof}
The condition $[A,J_1] = 0$ is automatically satisfied since $J$ is integrable on $\frak g$.

Suppose that $(J,g)$ is LCSKT with $A$ non-degenerate. From  Lemma \ref{Lem: alm-ab-LCSKT},  the eigenvalues of $A$, $k_1,...,k_{2n-2} $, come in pairs  and can be rearranged such that
$$Re(k_1)= ... = Re(k_{2l}) = -\frac{1}{2}(a+\lambda_{2n}),\quad Re(k_{2l+1})= ...= Re(k_{2n-2}) = 0.$$
Then $tr A = - l(a+\lambda_{2n})$. Now, taking traces in \eqref{Eq: alm-ab-matrix} one obtains that $tr (A^2+AA^t) = -(a+\lambda_{2n}) tr A$. Thus $\|S(A)\|^2= \frac{1}{2} l(a + \lambda_{2n})^2.$ This yields equality in Lemma\ref{Lem: normalop-estimate}, thus $A$ is a normal endomorphism i.e $[A,A^t] = 0$.

For the proof of the converse assertion we consider the matrix $B=(a+\lambda_{2n})A+A^2+A^tA$ in the left-hand side of  \eqref{Eq: alm-ab-matrix}. By hypothesis, $A$ is normal and a very simple computation shows that $B$ is also normal.  Recall that a normal matrix is skew-symmetric if and only if its eigenvalues are zero or purely imaginary. We make use of the spectral theorem for normal matrices to show that this is the case. The matrix $A$ is unitary conjugate to
$$\mathrm{diag}\left(-\frac{1}{2}(a+\lambda_{2n})+i \mathrm{Im}(k_1), -\frac{1}{2}(a+\lambda_{2n})+i \mathrm{Im}(k_2), \cdots, i \mathrm{Im}(k_{2n-3}), i\mathrm{Im}(k_{2n-2})\right).$$
Another computation will show that $B$ is unitary conjugate to
$$\mathrm{diag}(0,0, \cdots, i (a+\lambda_{2n}) \mathrm{Im}(k_{2n-3}),  i (a+\lambda_{2n}) \mathrm{Im}(k_{2n-2}))$$
and the result follows.

\end{proof}

In this situation  where $detA \neq 0,$ the non-vanishing  $1$-form $\a \in \frak g^*$  is given by
$$\a = \lambda_1 e^1 + \lambda_{2n}e^{2n} \neq 0,$$
with $a\lambda_1 = 0 $ and $\lambda_{2n} =-( 2Re(k)+ a),$ where $k$ is an eigenvalue of A.

We can see that even if $\a \neq 0,$ we get the  trivial   case  $(dH=0)$ by setting $\lambda_{2n}=0,$ $a=0$ and $\lambda_1 \neq 0$ (arbitrary), so
$$
(A^2 + A^t A) \in \mathfrak{so}(\frak n_1).
$$
Thus $\a = \lambda_1 e^1 \neq 0$, in this case the 3-form $H$ is degenerate.

A classification, up to isomorphism, of  $6$-dimensional  almost abelian Lie algebras  was given in \cite{M63}, corrected and completed by
\cite{SH}  (see also \cite{B,FP}).
By Corollary \ref{Cor: nilpotent-main}, the only 6-dimensional  nilpotent almost abelian Lie algebra  admitting a (trivial) LCSKT structure is  $\frak h_8 =(0,0,0,0,0,f^{12})$.

The six-dimensional non-nilpotent almost abelian Lie algebras admitting a complex structure are classified, up  to isomorphism,  in \cite[Theorem 3.2]{FP}
and they are  denoted  by $\frak l_i, i=1,\ldots,26$.

Using   the previous results,  we will   now  construct some explicit examples of   $6$-dimensional (non-nilpotent)  almost abelian Lie algebras admitting an LCSKT structure.  Recall that a Lie algebra $\frak g$  is unimodular if ${\mbox{tr} }(ad_X) =0$, for every $X \in \frak g$, and the unimodularity is a necessary condition for a Lie group to admit a lattice  \cite{JM}.

\begin{ex}
Consider the indecomposable  six-dimensional almost abelian Lie algebra
$$
l_8^{p,q,s}= \frak g_{6.8}^{p,q,q,s}=(pf^{16},qf^{26},qf^{36},sf^{46}+f^{56},sf^{56}-f^{46},0)$$ with $p \neq 0$ and $0 <  |q|\leq |p|$. The Lie algebra is  unimodular if  and only if  $s=-\frac{1}{2}(p+2q)$. Let $(J, g)$ be the almost  Hermitian structure given by
$$
J (f_1)= f_6,  J (f_2) = f_3, J (f_4) = f_5, \quad g=\sum_{k=1}^6 f^k \otimes f^k.
 $$
 Therefore
$$
J=\left(
    \begin{array}{ccc}
      0 & 0 & -1 \\
      0 & J_1 & 0 \\
      1 & 0 & 0 \\
    \end{array}
  \right)
, \quad
J_1 = \left(
        \begin{array}{cccc}
          0 &-1 & 0 & 0 \\
          1 & 0 & 0 & 0 \\
          0 & 0 & 0 &-1 \\
          0 & 0 & 1 & 0 \\
        \end{array}
      \right)
$$
and  the basis  $(f_1, \ldots, f_6)$  is adapted   to the  splitting
$$\frak l_8^{p,q,s} = \R_{f_1} \oplus \frak n_1 \oplus \R_{f_6}.$$
We have
$$B=(ad_{f_6})\vert_{\frak n} = \left(
                                 \begin{array}{cc}
                                   p & 0 \\
                                   0 & A \\
                                 \end{array}
                               \right),
$$  with
$$
A= \left(
  \begin{array}{cccc}
    q & 0 & 0 & 0 \\
    0 & q & 0 & 0 \\
    0 & 0 & s & 1 \\
    0 & 0 & -1 & s \\
  \end{array}
\right).
$$
Therefore,  $J$ is integrable,   since  $[A,J_1 ]=0$ and $J(\frak n_1) \subset \frak n_1,$ and the matrix $A$ is non-degenerate, since $q \neq 0$. Moreover,  we have also $[A,A^t]=0$, i.e. $A$ is normal. From Theorem \ref{Thm: alm-ab-main}, the Lie  algebra $\frak l_8^{p,q,s}$ admits an LCSKT structure  if either $s=q$ or $s=0$.   The non-vanishing 1-closed form  $\a$ is given by $\a=-(2q+p)f^6$  with  $0<|q|\leq |p|$ and $q \neq -\frac{p}{2}$.

We remark that, for generic parameters $p,q,s$ with $p\neq 0$ and $0<|q|\leq |p|$, the 3-form $H$ is given by $H=-2 (q f^{123}+s f^{145})$ and thus $dH= -2 q(2q+p) f^{1236} -2s(s+p) f^{1456}$.  In particular,  for the unimodular case $s = q = -\frac{p}{4}$, we have a non-trivial LCSKT structure with $dH = \frac{p^2}{4} (f^{1236}+ f^{1456}) \neq 0.$ Note that   $(J, g)$   is  LCB, since $v =0$ (see Remark \ref{LCBalmostab}).

\end{ex}

We now construct   an explicit example of  almost abelian  symplectic  solvmanifold $(M=\Gamma/G,J,g)$ endowed  with an  invariant   LCSKT  structure $(J,g)$.

In general  for a solvable Lie group it is  not easy  to find a lattice and to establish a general existence criterium, for more details see the references  \cite{B, CM12, AO}.  For almost abelian Lie groups there is the following sufficient criterion  (\cite[Proposition  2.1]{B}).

\begin{prop}[\cite{B}]
Let $G= \R \ltimes_{\varphi}\R^{m-1}$ be a $m$-dimensional  almost abelian Lie group. Then $G$ admits a lattice if and only if there exists a  real number $t_0 \neq 0$ for which $\varphi(t_0 )=exp(t_0 ad_{e_{2n}})$ can be conjugated to an integer matrix.
\end{prop}
In this case where there is a matrix P such that $P\varphi (t_0) P^{-1}$ is an integer matrix, a lattice is given by \cite{AO}
$$\Gamma_{t_0} =t_0 \mathbb{Z} \ltimes P^{-1}\mathbb{Z}^{2n-1}.$$
Recall that  a 6-dimensional Lie algebra  $\frak g$ is called symplectic, if there is a closed 2-form $\omega \in \Lambda^2 \frak g^*$ such that $\omega^3 \neq 0$.  The symplectic structure on $\frak g$ induces an invariant  symplectic structure on the quotient $\Gamma/G$.

\begin{ex}  Consider  the decomposable  six-dimensional  unimodular almost abelian Lie algebra
$${\frak l}_{23}^0= \frak g_{5.14}^0 \oplus \R=(f^{26},-f^{16},f^{46},0,0,0)$$
 which  corresponds  to the Lie group $(G_{5.14}^0 \times \R)$ (\cite{FP}, table 3, The Appendix). The Lie algebra  ${\frak l}_{23}^0$ admits the (integrable)  complex structure
  $$J(f_1 )=f_2,  \, J(f_3 )=f_5,   \, J(f_4)=f_6. $$
Applying the change of basis $e^1= -f^4$, $e^4=f^1$, $e^k=f^k$ for $k=2,3,5,6,$ we obtain
 the complex structure  given by
$$J(e_1)=e_6,  \, J(e_3)=e_5,  \, J(e_2 )=e_4.$$
Therefore  we  have  the usual orthonormal decomposition ${\frak l}_{23}^0= \R e_1  \oplus \frak n_1  \oplus   \R  e_6.$\\
If we consider the $J$-Hermitian metric $g = \sum_{i=1}^6 e^i \otimes e^i,$  the associated  fundamental 2-form $\Omega$ is  given  by $\Omega = f^{16} + f^{24} + f^{35}$. Moreover, by a direct computation , we have $H = e^{136}$  and thus  $dH  =0$.

By imposing the closure   of the 1-form $\a = \sum_{i =1}^6 \lambda_i e^i$  we get the  following constraints
$$\lambda_2=\lambda_4=\lambda_3=0, \lambda_1 \in \R.$$

By imposing the conditions \eqref{c2} and \eqref{c3}  we get  $\lambda_5=0.$
Consequently, the almost abelian Lie algebra $l_{23}^0$  admits a trivial LCSKT structure with  $\a=\lambda_1 e^1$,  for arbitrary $\lambda_1 \in \R-\{0\}$.

Moreover, by \cite[Theorem 1.1, Table 1]{CM12},  $exp(t ad_{e_6})$  is conjugated to an integer matrix, for $t=2\pi$ and $ t \in \{\pi,\frac{\pi}{2},\frac{\pi}{3}\}$, therefore  $(G_{5.14}^0 \times \R)$ admits lattices $\Gamma_t$.
As a consequence the solvmanifolds $M_t=(\Gamma_t/(G_{5.14}^0 \times \R),J,g)$ are LCSKT  and have first Betti number $b_1=5$  for $t=2\pi$  and $b_1=3$ for $t \in \{\pi,\frac{\pi}{2},\frac{\pi}{3}\}$. Moreover,  by  \cite[Theorem 1, Appendix B]{MM}  $\Gamma_t/(G_{5.14}^0 \times \R)$  has the invariant  symplectic form
$$\omega = \omega_{1,3}e^{13}+ \omega_{1,5}e^{15}+\omega_{1,6}e^{16}+\omega_{2,4}e^{24}+\omega_{2,6}e^{26}+\omega_{3,6}e^{36}+\omega_{4,6}e^{46}+\omega_{5,6}e^{56}$$
where $\omega_{2,4}(\omega_{1,3}\omega_{5,6}-\omega_{3,6}\omega_{1,5})\neq 0.$
Note that the 6-dimensional LCSKT almost abelian symplectic solvmanifolds $M_t=(\Gamma_t/(G_{5.14}^0 \times \R),J,g)$, are not balanced but they are LCB. Indeed,
by using (\ref{Eq: Lee}), the Lee form $\theta$  is given by $\theta =-e^5$, which is  closed.  Moreover,
the Chern-Ricci  form vanishes, but the Bismut Ricci form $\rho^B=- e^{16}$ is not zero.
\end{ex}

\section*{Acknowledgments}

The second author acknowledges the hospitality of the Department of Mathematics of the University of Torino during a one-month research stay which was funded by GNSAGA of  INdAM. The research of the second author was partially financed by Portuguese Funds through FCT (Fundação para a Ciência e a Tecnologia) within the projects UIDB/MAT/00013/2020 and UIBP/MAT/00013/2020. The third author was  supported by GNSAGA of  INdAM,  by the project PRIN 2017 \lq \lq Real and Complex Manifolds: Topology, Geometry and Holomorphic Dynamics” and by a grant from the Simons Foundation (\#944448).  The fourth author is grateful to Mohamed Hadji Brahim for  useful discussions. The authors would like also to thank  Fabio Paradiso and the anonymous referee for useful comments.


\begin{thebibliography}{12}


\bibitem{AI} B.  Alexandrov, S. Ivanov, Vanishing theorems on Hermitian manifolds,{ \it Differential
Geom. Appl. } {\bf 14} (2001), no. 3, 251--265.

\bibitem{AO} A. Andrada, M. Origlia, Lattices in almost abelian Lie groups with locally conformal K\"{a}hler or symplectic structures, {\it Manuscripta Math.}  {\bf 155} (2018), 389--417.

\bibitem{AU}  D. Angella, L. Ugarte, Locally conformal Hermitian metrics on complex non-K\"{a}hler manifolds,  {\em Mediterr. J. Math. }{\bf 13}  (2016), no. 4, 2105--2145.

\bibitem{AL}  R.M. Arroyo, R.A. Lafuente, The long-time behavior of the homogeneous pluriclosed flow, {\em Proc. Lond. Math. Soc.} (3) {\bf 119} (2019), no. 1, 266--289.

\bibitem{AN} R.M. Arroyo, M. Nicolini,  SKT structures on nilmanifolds,   {\em Math. Z.} {\bf 302} (2022), no. 2, 1307--1320. 

\bibitem{Bismut}  J.M. Bismut, A local index theorem for non K\"{a}hler manifolds, {\em Math. Ann.}  {\bf 284} (1989),
1057--1068.

\bibitem{B} C. Bock, On low dimensional solvmanifolds,  {\em Asian J. Math.}  {\bf 20}  (2016), 199--262 .

\bibitem{COUV} M. Ceballos, A. Otal, L. Ugarte, R. Villacampa, Invariant Complex Structures on 6-Nilmanifolds: Classification, Fr\"{o}licher Spectral Sequence and Special Hermitian Metrics, {\em J. Geom. Anal.} {\bf 26} (2016), 252--286.

\bibitem{CM12} S. Console, M. Macr\'i,  Lattices, cohomology and models of 6-dimensional almost abelian solvmanifolds,
{\em Rend. Semin. Mat. Univ. Politec. Torino}  {\bf 74} (2016), 95--119.

\bibitem{CFGU}  L.A. Cordero, M. Fernandez, A. Gray, L. Ugarte, Nilpotent complex structures on compact nilmanifolds, {\em Rend. Circ. Mat. Palermo} {\bf 49} Suppl. (1997), 83--100.

\bibitem{CFGU00} L.A. Cordero, M. Fernandez, A. Gray, L. Ugarte, Compact nilmanifolds with nilpotent complex structure: Dolbeault cohomology, {\em Trans. Amer. Math. Soc.}  {\bf 352} (2000), 5405--5433.

\bibitem{DF} I. Dotti, A. Fino, Hyperk\"ahler torsion structures invariant by nilpotent Lie groups, {\em Classical Quantum Gravity}  {\bf 19} (3) (2002), 551--562.

\bibitem{EFV} N. Enrietti, A. Fino,  L. Vezzoni, Tamed symplectic forms and strong K\"{a}hler with torsion metrics,  {\em J. Symplectic Geom.} {\bf  10}  (2012),  203--223.

\bibitem{FG} A. Fino, G. Grantcharov, Properties of manifolds with skew-symmetric torsion and special holonomy, {\em Adv. Math.} {\bf 189} (2004), no. 2, 439--450.

\bibitem{FOU} A. Fino, A. Otal, L. Ugarte, Six-dimensional solvmanifolds with holomorphically trivial canonical bundle, {\em Int. Math. Res. Not. }   {\bf 24} (2015), 13757--13799.

\bibitem{FP} A. Fino, F. Paradiso, Generalized K\" ahler Almost Abelian Lie Groups,  {\em Ann. Mat. Pura Appl.}  (4) {\bf 200}  (2021), 1781--1812.


\bibitem{FP20} A. Fino,  F.  Paradiso, Balanced Hermitian structures on almost abelian Lie algebras,  {\em J. Pure Appl. Algebra}  {\bf 227} (2023), no. 2, Paper No. 107186.

\bibitem{FP22} A. Fino,  F.  Paradiso,  Hermitian structures on a class of almost nilpotent solvmanifolds,  {\em J. Algebra}  {\bf 609} (2022), 861--925.

\bibitem{FPS} A. Fino, M. Parton, S. Salamon, Families of strong KT structures in six dimensions, {\em Comm. Math. Helv.}  {\bf 79} (2004) no. 2, 317--340.

\bibitem{FV} A. Fino, L. Vezzoni,  A correction to ``Tamed symplectic forms and strong K\"{a}hler with torsion metrics'',  {\it J. Symplectic Geom.}  {\bf 17} (2019), 1079--1081.

\bibitem{FAS} M. Freibert, A. Swann, Two-step solvable SKT shears,  {\em Math. Z.}  {\bf 299}  (2021), no. 3-4, 1703--1739.

\bibitem{FS22} M. Freibert, A. Swann,  Compatibility of balanced and SKT metrics on two-step solvable Lie groups, arXiv:2203.16638v1[math.DG].

\bibitem{GHR} S.J. Gates Jr.,  C.M. Hull, M.Ro\v{c}ek, Twisted multiplets and new supersymmetric nonlinear $\sigma$-models, {\it Nuclear Phys. B}  {\bf 248} (1984), 157--186.

\bibitem{G84} P. Gauduchon, La 1-forme de torsion d'une vari\'{e}t\'{e} hermitienne compacte,  {\em Math. Ann.} {\bf 267} (1984), 495--518.

\bibitem{G97} P. Gauduchon, Hermitian connections and Dirac operators, {\em Boll. Un. Mat. Ital.}  {\bf 11 B} (1997), no. 2 suppl., 257--288.

\bibitem{GM2} M. Gualtieri, Generalized K\"{a}hler geometry, {\em Commun. Math. Phys.}  {\bf 331} (2014), 297--331.

\bibitem{HP} P.S. Howe,  G.  Papadopoulos, Further remarks on the geometry of two-dimensional nonlinear $\sigma-$models, {\em Classical Quantum Gravity } {\bf 12} (1988), 1647--1661.

\bibitem{LUV} A. Latorre, L. Ugarte, R. Villacampa,  Fr\"olicher spectral sequence of compact complex manifolds with special Hermitian metrics, arXiv:2207.14669v1[math.DG].

\bibitem{LRV} J. Lauret, E.A. Rodriguez-Valencia, On the Chern-Ricci flow and its solitons for Lie groups, {\em Math. Nachr.} {\bf 288 } (2015), no. 13, 1512--1526.

\bibitem{LW} J. Lauret, C. Will,  On the symplectic curvature flow for locally homogeneous manifolds, {\em J. Symplectic Geom.} {\bf 15} (2017), no.1, 1--49.

\bibitem{MM} M. Macr\`{\i}, Cohomological properties of unimodular six dimensional solvable Lie algebras,  {\em Differential Geom. Appl.}  {\bf 31}  (2013),  112--129.

\bibitem{MS} T.B. Madsen, A. Swann,  Invariant strong KT geometry on four-dimensional solvable Lie groups, {\em  J. Lie Theory} {\bf 21} (2011), 55--70.

\bibitem{AM45} A. Malcev, On solvable Lie algebras, {\em Izv. Akad. Nauk SSSR Ser. Mat.}  {\bf 9} (1945), 329--356.

\bibitem{JM} J. Milnor, Curvature of left invariant metrics on Lie groups, {\em Adv.  Math.} {\bf 21} (1976), 293--329.

\bibitem{M63} G.M. Mubarakzyanov, Classification of solvable Lie algebras of sixth order with a non-nilpotent basis element (Russian), {\em  Izv. Vyssh. Uchebn. Zaved. Mat.}  {\bf 4} (1963), 104--116.

\bibitem{OOS} L. Ornea, A.  Otiman, M.  Stanciu, Compatibility between non-Kähler structures on complex (nil)manifolds,  arXiv:2003.10708v2 [math.DG], to appear in {\em Transform. Groups}.

\bibitem{P} F. Paradiso, Locally conformally balanced metrics on almost abelian Lie algebras,   {\em Complex Manifolds} {\bf  8} (2021), no. 1, 196--207.

\bibitem{MSR} M.S. Raghunathan, Discrete Subgroups of Lie Groups, Springer (1972).

\bibitem{S} S. Salamon, Complex structures on nilpotent Lie algebras, {\em J. Pure Appl. Algebra} {\bf 157}  (2001), 311--333.

\bibitem{SH} A. Shabanskaya, Classification of six dimensional solvable indecomposable Lie algebras with a codimension one nilradical over $\R$, Shabanskaya,  Thesis (Ph.D.)--The University of Toledo. 2011. 210 pp. ISBN: 978-1124-69251-7

\bibitem{Stro86}  A. Strominger,  Superstrings with torsion, {\em  Nucl. Phys. B}  {\bf 274}  (1986), 253–284.

\bibitem{TW}  V. Tosatti, B. Weinkove,  The Chern-Ricci flow, {\em Atti Accad. Naz. Lincei Rend. Lincei Mat. Appl.}  {\bf 33} (2022), no.1, 73--107.

\bibitem{Ug} L.Ugarte, Hermitian structures on six-dimensional nilmanifolds, {\em Transform. Groups} {\bf 12} (2007), 175--202.

\bibitem{V1} I. Vaisman, On locally conformal almost K\"{a}hler manifolds, {\em Israel J. Math.}  {\bf 24} (1976),   338--351.

\bibitem{V2} I. Vaisman, On locally and globally conformal K\"{a}hler manifolds, {\em Trans. Amer. Math. Soc.} {\bf 262} (1980), 533--542.

\bibitem{Yano}  K. Yano, Differential geometry on complex and almost complex spaces,  {\em Pergamon Press} (1965).


\end{thebibliography}
\end{document}